\documentclass[12pt]{amsart}
\usepackage[utf8]{inputenc}
\usepackage[english]{babel}
\usepackage{amssymb}
\usepackage{amsthm}
\usepackage{amsmath}
\usepackage{verbatim}
\usepackage{graphicx}
\usepackage{pgf,tikz,pgfplots}
\pgfplotsset{compat=1.15}
\usepackage{mathrsfs}
\usetikzlibrary{arrows}
\usepackage{caption}
\usepackage{hyperref}

\usepackage{hyperref}
\theoremstyle{definition}
\newtheorem{mydef}{Definition}[section]

\theoremstyle{plain}
\newtheorem{thm}[mydef]{Theorem}
\newtheorem{all}[mydef]{Statement}
\newtheorem{cl}[mydef]{Claim}
\newtheorem{lemma}[mydef]{Lemma}

\newtheorem*{jel}{Notation}

\newcommand\cC{{\mathcal C}}
\newcommand\cF{{\mathcal F}}
\newcommand\cn{{\mathfrak{C}_n}}

\title{Chain-dependent Conditions in Extremal Set Theory}

\author{D\'aniel T. Nagy}
\address{Alfr\'ed R\'enyi Institute of Mathematics}
\email{nagydani@renyi.hu}
\thanks{D.T. Nagy's research is partially supported by NKFIH grants FK 132060 and PD 137779 and by the J\'anos Bolyai Research Fellowship of the Hungarian Academy of Sciences}

\author{Kartal Nagy}
\address{E{\"o}tv{\"o}s Lor\'and University}
\email{kartal97@student.elte.hu}

\begin{document}

%\begin{center}
	%{\LARGE \textsc{chain-dependent conditions in extremal set theory}} \\
	%\vspace{0.4cm}
%{\large {Dániel T. Nagy$^a$ and Kartal Nagy$^b$}} \\[0.3cm]
	%{\normalsize {$^a$Alfréd Rényi Institute of Mathematics, $^b$Eötvös Loránd University}} \\[0.2cm]
%{\normalsize Budapest, 2021}
%\end{center}

\maketitle

\begin{abstract}
In extremal set theory our usual goal is to find the maximal size of a family of subsets of an $n$-element set satisfying a condition. A condition is called chain-dependent, if it is satisfied for a family if and only if it is satisfied for its intersections with the $n!$ full chains. We introduce a method to handle problems with such conditions, then show how it can be used to prove three classic theorems. Then, a theorem about families containing no two sets such that $A\subset B$ and $\lambda \cdot |A| \le |B|$ is proved. Finally, we investigate problems where instead of the size of the family, the number of $\ell$-chains is maximized. Our method is to define a weight function on the sets (or $\ell$-chains) and use it in a double counting argument involving full chains.

%\textcolor{red}{MEG DOLGOZOK RAJTA}
\end{abstract}

\section{Introduction}

\begin{jel}
We will denote the set $\{ 1,2,\dots,n \}$ by $[n]$, the family of the subsets of $[n]$ by $2^{[n]}$. %and the family of the $i$-element subsets of $[n]$ by $\displaystyle\binom{[n]}{i}$. If $i > n$, then this notation means the empty set.
\end{jel}

Sperner's Theorem \cite{sperner} is one of the central results in extremal set theory. It describes the largest possible families of finite sets none of which contains any other set in the family. It is named after Emanuel Sperner, who published it in 1928.
%\cite{sperner}

\begin{thm}\label{sperner}
%We can select at most $ \displaystyle\binom{n}{\lfloor \frac{n}{2} \rfloor} $ sets from $2^{[n]}$, so that there do not exist selected sets $A$ and $B$, for which $A \subset B$.
Let $\cF\subset 2^{[n]}$ be a family such that there are no two sets $A,B\in\cF$ satisfying $A\subset B$. Then $|\cF|\le \displaystyle\binom{n}{\lfloor \frac{n}{2} \rfloor}$.
\end{thm}

There are many generalizations and variants to Theorem \ref{sperner}. We will mention just two of these that are relevant to our topic:

\begin{thm}\label{erdos}
%\emph{(Erdős, (1945) \cite{erdos}):} We can select at most $\sum_{i = \lfloor \frac{n-k}{2} \rfloor}^{\lfloor \frac{n+k}{2} \rfloor} \displaystyle\binom{n}{i}$ sets from  $2^{[n]}$, if we cannot select sets $A$ and $B$ such that $A \subset B$ and $ \vert B - A \vert > k$. %\cite{erdos}
\emph{(Erdős, (1945) \cite{erdos}):} Let $\cF\subset 2^{[n]}$ be a family such that there are no two sets $A,B\in\cF$ satisfying $A\subset B$ and $\vert B - A \vert > k$. Then $$|\cF|\le \sum_{i = \lfloor \frac{n-k}{2} \rfloor}^{\lfloor \frac{n+k}{2} \rfloor} \displaystyle\binom{n}{i}.$$
\end{thm}

\begin{thm}\label{katona}
%\emph{(Katona, (1971) \cite{katona}):} We can select at most $\sum_{i \equiv \lfloor \frac{n}{2} \rfloor (mod \hspace{1mm} k)} \displaystyle\binom{n}{i}$ sets from $2^{[n]}$, if we cannot select two different sets $A$ and $B$ such that $A \subset B$ and $\vert B - A \vert \leq k$. %\cite{katona}
\emph{(Katona, (1971) \cite{katona}):} Let $\cF\subset 2^{[n]}$ be a family such that there are no two sets $A,B\in\cF$ satisfying $A\subset B$ and $\vert B - A \vert < k$. Then $$|\cF|\le \sum_{i \equiv \lfloor \frac{n}{2} \rfloor \pmod{k}} \displaystyle\binom{n}{i}.$$
\end{thm}

Griggs \cite{griggs88} found an alternate proof for the above theorem, describing the cases of equality as well.

These theorems maximize the size of a family of subsets of $[n]$ satisfying a certain condition $D$. In the present paper we will show what is the common property of these conditions, and then prove a theorem that provides a method to solve all such problems.
%However these conditions have a common property. The goal of the present paper is to find general form of these properties.  %are searching the answer for theorems, that what is the maximum number of sets that we can choose from $2^{[n]}$, such that a condition $D$ satisfied for the selected family.

\begin{mydef}
We call a family $\mathcal{C} = \{A_0, A_1, \dots A_n \}$ a {\it full chain} if $A_0$, $A_1 \dots A_n \in 2^{[n]}$ are such that $\vert A_i \vert = i$, and $A_i \subset A_{i+1}$ for every $0 \leq i < n$. The family of all $n!$ full chains in $2^{[n]}$ is denoted by $\cn$.
\end{mydef}

\begin{mydef}
We say that a condition $D$ is a {\it chain-dependent} if $D$ is satisfied for a family $\mathcal{F} \subset 2^{[n]}$ if and only if it is satisfied for all families $\mathcal{F} \cap \mathcal{C}$, where $\cC\in\cn$ is a full chain.
\end{mydef}

For example, the property of being an antichain is chain-dependent.

In Section 2, we will prove our main theorem that will help us to find an upper bound on the size of a family in $2^{[n]}$ that satisfies a condition $D$. In Section 3, it will be used to find the maximal size of such a family in the case of various chain-dependent conditions, namely to reprove Theorems \ref{sperner}, \ref{erdos} and \ref{katona}. In Section 4 we consider another chain-dependent condition, and prove a new theorem about families containing no two sets such that $A\subset B$ and $\lambda \cdot |A| \le |B|$.

In extremal set theory, after the problem of finding the largest family in $2^{[n]}$ satisfying a certain property is solved, a common way to move forward is to determine the maximum number of $\ell$-chains $G_1\subset \dots \subset G_\ell$ in a family satisfying the same property. See \cite{lchain}, \cite{containments}, \cite{genforb} for examples of such results. In Section 5 we generalize Theorem \ref{erdos} in such a way and discuss a possible generalization of Theorem \ref{katona}.

%In Section 5, we investigate problems where instead of the size of the family, we have to maximize the number of $\ell$-chains of given size $\ell$ in the family.
%After this, we will answer the following question: what is the maximum number of sets that we can choose from $2^{[n]}$, such that the chain-dependent condition $D$ satisfied for the family.

\section{Main theorem and proof}

%We need the following two definitions:

%\begin{mydef}
%Let $\mathcal{F} \subset 2^{[n]}$ be a family and let $\mathcal{C}$ be a full chain. Then $$s(n,\mathcal{F},\mathcal{C}):= \sum_{i: \exists F \in \mathcal{F} \cap \mathcal{C}, \vert F \vert = i} \displaystyle\binom{n}{i}.$$
%\end{mydef}

\begin{jel}
Let $\mathcal{F} \subset 2^{[n]}$ be a family and $D$ a condition. Then $\mathcal{F} \rightarrowtail D$ means that $\mathcal{F}$ satisfies the condition $D$.
\end{jel}

We introduce a weight function as follows:

%\begin{mydef}
%Let $D$ be a condition. Then $$ S(n,D):=\max_{\mathcal{F} \rightarrowtail D} s(n,\mathcal{F},\mathcal{C}).$$
%\end{mydef}

\begin{mydef}
Let $\cF\subset 2^{[n]}$ be a family. For a set $G\subset [n]$, let $\omega_{\cF} (G):= \displaystyle\binom{n}{\vert G \vert}$ if $G \in \mathcal{F}$, and $\omega_{\cF} (G):= 0$ otherwise.
We define
$$s(\cF):=\max_{\cC\in \cn}  \sum_{G \in \mathcal{C}} \omega_{\cF} (G).$$
\end{mydef}

\begin{mydef}
Let $D$ be a condition. Then $$ S(n,D):=\max_{\cF\subset 2^{[n]},~\cF \rightarrowtail D} s(\mathcal{F}).$$
\end{mydef}

In other words, $S(n,D)$ is defined as the maximum possible weight of a full chain over all $\cF$ satisfying $D$.

\begin{thm} \label{cdt}
Let $n$ be a positive integer and $\mathcal{F} \subset 2^{[n]}$. If $\mathcal{F} \rightarrowtail D$ holds for a condition $D$, then $\vert \mathcal{F} \vert \leq S(n,D)$. %Moreover, if $D$ is a chain-dependent condition, then the upper bound is the best possible.
\end{thm}

\begin{proof}

Obviously %Take a full chain $c$. %Let $F$ be the smallest set (if it exists), for which $F \in c$ and $F \in \mathcal{F}$. If $\vert F \vert = k$, then the biggest set which is in $\mathcal{F}$ and in $c$ could contain at most $\lambda k$ elements. Now we can state this inequality:

\begin{equation}\label{Somega}
    \sum_{G \in \mathcal{C}} \omega_{\cF} (G) \leq s(\cF) \leq  S(n,D)
\end{equation}

holds for any full chain $\cC\in\cn$. We can write this inequality:

$$ S(n,D) \geq \frac{1}{n!} \sum_{\cC\in\cn} \left( \sum_{G \in \cC} \omega_{\cF} (G) \right) = \frac{1}{n!} \sum_{G \in \cF} \left( \sum_{\cC \in \cn:~ G \in \cC} \omega_{\cF} (G) \right) = $$

$$ \frac{1}{n!} \sum_{G \in \mathcal{F}} \vert G \vert !  \left(n- \vert G \vert \right) ! \omega_{\cF}  (G) = \sum_{G \in \mathcal{F}} 1 = \vert \mathcal{F} \vert.$$

The first inequality follows from (\ref{Somega}), since $|\cn| = n!$. At the third step we use the fact that $k! (n-k)!$ full chains go through any given set of size $k$. Altogether we get an upper bound for $|\mathcal{F}|$.
\end{proof}

While Theorem \ref{cdt} holds for any condition $D$, it is particularly useful for chain-dependent ones, since determining the value of $S(n,D)$ is usually easy for them. We also note that the weight function $\omega$ was used previously in \cite{dtnagy} and \cite{gerbner} to solve problems concerning specific forbidden structures without a general approach.

Our method is similar to Lubell's \cite{lubell} proof of Sperner's Theorem that gave an upper bound on $|\cF|$ by examining the number of sets $\cF$ contains from a random full chain. (See \cite{griggsli} for applications of this technique.) By adding different weights to sets of different size, we can bound $|\cF|$ more precisely, and prove exact bounds even if the extremal family contains sets whose size is far from $\frac{n}{2}$.

\section{New proofs for earlier theorems}

In this section we will show how to use Theorem \ref{cdt} to prove Theorems \ref{sperner}, \ref{erdos} and \ref{katona}, all of which are about families satisfying a chain-dependent condition.

First, let us consider Theorem \ref{sperner}. Let $D$ denote the condition ``does not contain two different sets $A,B$ such that $A\subset B$'' or equivalently ``is an antichain''. If $\cF \subset 2^{[n]}$ is an antichain, then any full chain $\cC\in\cn$ contains at most one set from $\cF$. The weight $\omega_{\cF}(F)$ of this single set $F$ is at most $\displaystyle\binom{n}{\lfloor \frac{n}{2} \rfloor}$. Since this applies to all antichains $\cF$ and full chains $\cC$, we have $S(n,D)\le\displaystyle\binom{n}{\lfloor \frac{n}{2} \rfloor}$. Then Theorem \ref{cdt} implies $|\cF|\le\displaystyle\binom{n}{\lfloor \frac{n}{2} \rfloor}$. The corresponding lower bound is trivially given by the antichain formed by all sets of size $\lfloor \frac{n}{2} \rfloor$.

In Theorem \ref{erdos}, the condition $D$ is that the family $\cF$ contains no two sets such that $A\subset B$ and $|B - A|>k$. This implies that a full chain contains at most $k+1$ sets of $\cF$. Their total weight is at most the sum of the $k+1$ largest binomial coefficients, therefore
$$S(n,D)\le \sum_{i = \lfloor \frac{n-k}{2} \rfloor}^{\lfloor \frac{n+k}{2} \rfloor} \displaystyle\binom{n}{i}.$$
Theorem \ref{cdt} implies that this sum is an upper bound for $|\cF|$ as well. The corresponding lower bound on $|\cF|$ is given by taking all sets of size between $\lfloor \frac{n-k}{2} \rfloor$ and $\lfloor \frac{n+k}{2} \rfloor$.

Finally, in Theorem \ref{katona}, the condition $D$ states that $\cF$ contains no two sets $A,B$ such that $A\subset B$ and $|B - A| < k$. This means that the maximum total weight of the sets in a full chain is given by the maximum of the expression $\displaystyle\sum_{i=1}^t\binom{n}{a_i}$, where $a_{i+1}\ge a_{i}+k$ for all $1\le i \le t-1$.

Finding this is a simple task (as can be seen in the next lemma) and results in the value given in Theorem \ref{katona}. Theorem \ref{cdt} implies it is an upper bound for $|\cF|$ as well. The corresponding lower bound on $|\cF|$ is given by taking all sets of size $a_i,~i=1,2,\dots, t$ where the numbers $a_i$ are the values maximizing the binomial sum.

\begin{mydef}
Define $\mathcal{H} \subseteq 2^{\{0, ..., n\}}$ as follows: a set $H\subseteq \{0, ..., n\}$ is in $\mathcal{H}$ if and only if $\vert h-h' \vert \geq k$ for all different elements $h,h'\in H$. Let $f(H) = \displaystyle\sum_{h \in H}\binom{n}{h}$ for a set $H \in \mathcal{H}$ and let $f(\mathcal{H})= \max\limits_{H \in \mathcal{H}} \{ f(H) \}$.
\end{mydef}

\begin{lemma}
%For a set $H\subset \{0,1,\dots, n\}$, $f(\mathcal{H})= f(H)$ if and only if $H={\{i\in [0,n] ~|~i \equiv \lfloor \frac{n}{2} \rfloor \pmod{k}}\} ~\text{or}~ \{i\in [0,n] ~|~i \equiv \lceil \frac{n}{2} \rceil \pmod{k}\}$.
$f({\{i\in [0,n] ~|~i \equiv \lfloor \frac{n}{2} \rfloor \pmod{k}}\})=f(\mathcal{H})$.
\end{lemma}

\begin{proof}
The Lemma is trivial for $n<k$. Assume that $n\ge k$.

\begin{cl}
If $f(H)=f(\mathcal{H})$ for a set $H \in \mathcal{H}$, then $H=\{h_1,h_2,\dots, h_t\}$ satisfies the following conditions:
\begin{enumerate}
    \item $h_1<k$,
    \item $h_{i+1}-h_i=k$ for all $1\leq i<t$,
    \item $h_t>n-k$.
\end{enumerate}
\end{cl}

\begin{proof}
If the first or third case is not satisfied, then we can extend $H$ with $0$ or $n$. If the second case is not satisfied, then there exists a value $i$ such that $h_{i+1}-h_i>k$, so we can increase $f(H)$ by changing $h_i$ to $h_i+1$ or $h_{i+1}$ to $h_{i+1}-1$.
\end{proof}

%So we showed for all $n$ that $f(H)=f(\mathcal{H})$ only if there exists $h_m \in H$ such that $\frac{n-k}{2} \leq h_m < \frac{n+k}{2}$ and $x \in H \Leftrightarrow (x \equiv h_m \pmod{k}$ and $0 \leq x \leq n$). So in the remaining part we work with these sets only. For a fixed $k$, let $H_a^n$ be the set of this form for which $a= \frac{n}{2}-h_m $.

The above implies that if $f(H)=f(\mathcal{H})$ holds, then $H$ consists of all numbers of a mod $k$ residue class in $[0,n]$. If $k=2$, then the binomial sum is equal to $2^{n-1}$ for both residue classes. From now on, we will assume $k\ge 3$.
For a number $-\frac{k}{2}< a\le \frac{k}{2}$, such that $\frac{n}{2}+a \in \mathbb{Z}$ let
$$H_a^n:=\left\{i\in [0,n] ~|~i \equiv \frac{n}{2}+a \pmod{k}\right\}.$$

\begin{cl}
If $|a|<|b|\le\frac{k}{2}$, and $\frac{n}{2}+a, \frac{n}{2}+b \in \mathbb{Z}$ then $f(H_a^n) > f(H_b^n)$.
\end{cl}

\begin{proof}
We prove this claim by induction on $n$. It is true when $n=k$.
Supposing the statement holds for $n-1$, let us prove it for some $n>k$.

Since $f(H_a^n)=f(H_{-a}^n)$, we can assume that $0\le a<b\le\frac{k}{2}$. From the binomial identities we get
$f \left( H_a^{n} \right) =f \left( H_{a-\frac{1}{2}}^{n-1} \right) +f \left( H_{a+\frac{1}{2}}^{n-1} \right)$ and $f \left( H_b^{n} \right) =f \left( H_{b-\frac{1}{2}}^{n-1} \right) +f \left( H_{b+\frac{1}{2}}^{n-1} \right)$. The inductive hypothesis implies that $f \left( H_{a-\frac{1}{2}}^{n-1} \right)\ge f \left( H_{b-\frac{1}{2}}^{n-1} \right)$ and $f \left( H_{a+\frac{1}{2}}^{n-1} \right)\ge f \left( H_{b+\frac{1}{2}}^{n-1} \right)$ and at least one of the inequalities is strict. (If $b=\frac{k}{2}$, we have to notice $f \left( H_{\frac{k}{2}+\frac{1}{2}}^{n-1} \right)=f \left( H_{\frac{k}{2}-\frac{1}{2}}^{n-1} \right)$ before the hypothesis can be used.)

%We know that if $\frac{n+1}{2}+a \in \mathbb{Z}$, then $f \left( H_a^{n+1} \right) =f \left( H_{a+\frac{1}{2}}^n \right) +f \left( H_{a-\frac{1}{2}}^n \right)$.

%So if $\vert a \vert < \vert b \vert$, then $f \left( H_a^{n+1} \right)= f \left( H_{a+\frac{1}{2}}^n \right) + f \left( H_{a-\frac{1}{2}}^n \right) \geq f \left( H_{b+\frac{1}{2}}^n \right) + f \left( H_{b-\frac{1}{2}}^n \right) = f \left( H_b^{n+1} \right)$ follows from the inductive hypothesis.
\end{proof}

%\begin{mj}
%This formula is not true when $a=\frac{k}{2}$, because $f \left( H_{\frac{k}{2}}^{n+1} \right) = 2 \cdot f \left( H_{ k-\frac{1}{2}}^n \right)$, but it is easy to see that the inequality is remains true.
%\end{mj}

We proved that $f(H)$ is maximized if $H=H_0^n$ (for even $n$) or $H\in\{H_{-\frac{1}{2}}^n, H_{\frac{1}{2}}^n\}$ (for odd $n$).
\end{proof}

\section{Forbidding \texorpdfstring{$c$}{c}-times bigger sets}

What is the maximum size of a family from $2^{[n]}$ that contains no two different sets $A$ and $B$ such that $A \subset B$ and $c \cdot \vert A \vert \leq \vert B \vert$? This question was asked by G.O.H. Katona which eventually led to this paper.

In an earlier article \cite{sajat}, this was solved when $c>1$ is an integer:

\begin{thm}\label{main}
Let $n$ and $c>1$ be positive integers and $k= \left\lfloor \frac{n}{c+1} \right\rfloor$. %If we would like to choose the maximum possible number of sets from $2^{[n]}$ without choosing two different sets $A$ and $B$ such that $A \subset B$ and $c \cdot \vert A \vert \leq \vert B \vert$, then we should choose the sets which have at least $k+1$ and at most $c (k+1)-1$ elements.
Then the largest family $\cF\subset 2^{[n]}$ that contains no two different sets $A,B\in\cF$ such that $A \subset B$ and $c \cdot \vert A \vert \leq \vert B \vert$ is
$$\left\{F\subset [n]~:~k+1\le |F| \le c(k+1)-1\right\}.$$
\end{thm}

Now we will prove a generalization of Theorem \ref{main} using Theorem \ref{cdt}. We will show what the answer is, if we change the integer $c$ to a real number $\lambda > 1$.

First of all, we will have to use these two definitions for the theorem:

\begin{mydef}
Let $1 \leq k \leq n$ be an integer and $1<\lambda \in \mathbb{R}$. Then $$t(n,k,\lambda):= \sum_{k \leq i < \lambda k } \displaystyle\binom{n}{i}.$$
\end{mydef}

\begin{mydef}
Let $1<\lambda \in \mathbb{R}$. Then $$ T(n,\lambda):=\max_{1 \leq k \leq n} t(n,k,\lambda).$$
\end{mydef}

\begin{thm}
Let $n$ be a positive integer, $\lambda>1$ a real number, and $\mathcal{F} \subset 2^{[n]}$. Suppose that $\mathcal{F}$ contains no two members $A,B \in \mathcal{F} $ for which $A \subset B$ and $\lambda \cdot \vert A \vert \leq \vert B \vert$ holds. Then $\vert \mathcal{F} \vert \leq T(n,\lambda)$ and this bound is sharp.
\end{thm}

\begin{proof}

%We can check the following claim easily:

%\begin{cl} \label{claim}
%The condition of the theorem is a chain-dependent condition.
%\end{cl}

%\begin{kov}
%By the Claim \ref{claim}, we can use the Theorem \ref{cdt}, so $\vert \mathcal{F} \vert \leq T(n,\lambda)$.
%\end{kov}

Let $D$ denote the condition described in the theorem. Let $\cF\subset 2^{[n]}$ be a family that satisfies $D$, and $\cC\in\cn$ be a full chain. If $\cF'=\cF\cap \cC\not=\emptyset$, then let $k$ denote the size of the smallest element of $\cF'$. The condition implies that the size of the elements of $\cF'$ must be in the interval $[k, \lambda \cdot k)$. Therefore
$$\sum_{G \in \cC} \omega_{\cF}(G)=\sum_{G\in\cF'}\binom{n}{|G|}\le t(n,k,\lambda)\le T(n,\lambda).$$
Since the above holds for all $\mathcal{F} \rightarrowtail D$ and $\cC\in\cn$, we get that $S(n,D)\le T(n,\lambda)$. Then Theorem \ref{cdt} implies $|\cF|\le T(n,\lambda)$.

%Altogether we get an upper bound for $\mathcal{F}$.
We can show that this bound is sharp. If we choose the $k$ for which the $t(n,k,\lambda)$ is maximal, we can select select the family consisting of sets whose size is in the interval $[k, \lambda \cdot k)$. It is a good family whose size is $T(n,\lambda)$ which is our upper bound.
\end{proof}

\section{Maximizing the number of \texorpdfstring{$\ell$}{l}-chains}

\begin{mydef}
We call a family $\mathcal{G}_\ell=\{G_1 ,..., G_\ell\}$ whose elements satisfy $G_1 \subset ... \subset G_\ell$ an $\ell$-chain.
\end{mydef}

In this section we look at what happens when we want to maximize the number of $\ell$-chains in the family instead of its size. A theorem similar to Theorem \ref{cdt} is proved.

We introduce a weight function as follows:

\begin{mydef}
Let $\cF\subset 2^{[n]}$ be a family. For an $\ell$-chain $\mathcal{G}_\ell \subset 2^{[n]}$, let $\omega_{\cF} (\mathcal{G}_\ell):=
\displaystyle\binom{n}{\vert G_\ell \vert} \displaystyle\binom{\vert G_\ell \vert}{\vert G_{\ell-1} \vert} \dots \displaystyle\binom{\vert G_2\vert}{\vert G_1 \vert}$ if $\mathcal{G}_\ell \subset \mathcal{F}$, and $\omega_{\cF} (\mathcal{G}_\ell):= 0$ otherwise.
We define
$$s_{\ell}(\cF):=\max_{\cC\in \cn}  \sum_{\mathcal{G_\ell} \subset \mathcal{C}} \omega_{\cF} (\mathcal{G}_\ell).$$
\end{mydef}

\begin{mydef}
Let $D$ be a condition. Then $$ S_{\ell}(n,D):=\max_{\cF\subset 2^{[n]},~\cF \rightarrowtail D} s_{\ell}(\mathcal{F}).$$
\end{mydef}

In other words, $S_{\ell}(n,D)$ is defined as the maximum possible weight of a full chain if $\cF$ satisfies $D$.

\begin{thm} \label{cdt2}
Let $n$ be a positive integer and $\mathcal{F} \subset 2^{[n]}$. If $\mathcal{F}$ satisfies a condition $D$, then $\mathcal{F}$ contains at most $S_{\ell}(n,D)$ $\ell$-chains. %Moreover, if $D$ is a chain-dependent condition, then the upper bound is the best possible.
\end{thm}

\begin{proof}

Obviously %Take a full chain $c$. %Let $F$ be the smallest set (if it exists), for which $F \in c$ and $F \in \mathcal{F}$. If $\vert F \vert = k$, then the biggest set which is in $\mathcal{F}$ and in $c$ could contain at most $\lambda k$ elements. Now we can state this inequality:

\begin{equation}\label{Somega2}
    \sum_{\mathcal{G}_\ell \subset \mathcal{C}} \omega_{\cF} (\mathcal{G}_\ell) \leq s_{\ell}(\cF) \leq  S_{\ell}(n,D)
\end{equation}

holds for any full chain $\cC\in\cn$. We can write this inequality:

$$ S_{\ell}(n,D) \geq \frac{1}{n!} \sum_{\cC\in\cn} \left( \sum_{\mathcal{G}_\ell\subset \cC} \omega_{\cF} (\mathcal{G}_\ell) \right) = \frac{1}{n!} \sum_{\mathcal{G}_\ell \subset \cF} \left( \sum_{\cC \in \cn:~\mathcal{G}_\ell \subset \cC} \omega_{\cF} (\mathcal{G}_\ell) \right) = $$

$$ \frac{1}{n!} \sum_{\mathcal{G_\ell} \subset \mathcal{F}} \left(n- \vert G_\ell \vert \right) ! \left(\vert G_\ell \vert - \vert G_{\ell_1} \vert \right) ! \dots \left(\vert G_2 \vert - \vert G_1 \vert \right) ! \vert G_1 \vert ! \omega_{\cF}  (\mathcal{G_\ell}) =\sum_{\mathcal{G_\ell} \subset \mathcal{F}} 1$$

The first inequality follows from (\ref{Somega2}), since $|\cn| = n!$. At the third step we use the fact that $\left(n- \vert G_\ell \vert \right) ! \left(\vert G_\ell \vert - \vert G_{\ell_1} \vert \right) ! \dots \left(\vert G_2 \vert - \vert G_1 \vert \right) ! \vert G_1 \vert !$ full chains go through all sets of an $\ell$-chain formed by sets of size $\vert G_1 \vert,\dots,\vert G_\ell\vert$.
In the end we get an upper bound for the number of $\ell$-chains in $\mathcal{F}$.
\end{proof}

%In the previous section we were examined some theorems, what is the maximum size set which satisfies the given condition. In this section, we examine the question, that we get the same set if we want to maximize the number of $\ell$-chains instead of the number of elements.

Now we consider variants of the earlier problems where we have to maximize the number of $\ell$-chains instead of the number of elements. We investigate whether the extremal families will be the same or not.

First we see the condition of Erdős' theorem. We can prove that the same families will be optimal.

\begin{thm}
Consider a family $\cF\subset 2^{[n]}$ such that there are no two sets $A,B\in\cF$ satisfying $A\subset B$ and $\vert B - A \vert > k$ and let $\ell\le k+1$. %If $\ell \leq k$, then we get the maximal number of $\ell$-chains such that $\mathcal{G}_\ell \subset \mathcal{F}$ if $H \in \mathcal{F} \Longleftrightarrow \lfloor \frac{n-k}{2} \rfloor \leq \vert H \vert \leq \lfloor \frac{n+k}{2} \rfloor$ or $ \lceil \frac{n-k}{2} \rceil \leq \vert H \vert \leq \lceil \frac{n+k}{2} \rceil$.
Among such families, the ones containing the most $\ell$-chains are $\left\{F\subset [n]~|~\lfloor \frac{n-k}{2} \rfloor \leq \vert F \vert \leq \lfloor \frac{n+k}{2} \rfloor\right\}$ and $\left\{F\subset [n]~|~\lceil \frac{n-k}{2} \rceil \leq \vert F \vert \leq \lceil \frac{n+k}{2} \rceil\right\}$. (These two are the same if $n+k$ is even.)
\end{thm}

\begin{proof}
%Since it is a chain-dependent condition, the optimal set the is union of some full levels. The question is which levels give the maximal number of $\ell$-chains.
Since we have a chain-dependent-condition we can follow a similar strategy as before. Let $H \subset \{0, 1, \dots, n\}$ be a set such that there are no two elements $c_i, c_j \in H$ satisfying $\vert c_i - c_j \vert > k$.
Consider a full chain $\cC$ and assume that its set of size $i$ is in $\cF$ if and only if $i\in H$. We want to find the set $H$ for which the total weight of the $\ell$-chains within $\mathcal{C}$ is maximal. By Theorem \ref{cdt2}, this will be an upper bound for the number of $\ell$-chains in $\cF$. This bound will be sharp as we can take the family that is the union of full levels whose sizes are given by the optimal $H$.

The expression we want to maximize is
$$\omega_{\cF}(H):=\sum_{\mathcal{G}_\ell \subset \mathcal{C}} \omega_{\cF} (\mathcal{G}_\ell) = \sum_{c_1 < \dots < c_\ell \in H} \binom{n}{c_\ell} \binom{c_\ell}{c_{\ell-1}} \dots \binom{c_2}{c_1} =$$
\begin{equation}\label{c1c2}
\sum_{c_1 < \dots < c_\ell \in H} \frac{n!}{(n-c_\ell)!(c_\ell-c_{\ell-1})!\dots (c_2-c_1)!c_1!}.
\end{equation}

\begin{jel}
Let $H_i = \{ i, i+1, \dots, i+k \}$.
\end{jel}

If $min\{H \}=i$ then $max \{ H \} \leq i+k$, so $H \subset H_i$, therefore the expression \eqref{c1c2} takes its maximum in one of the sets $H_i$. So one has to determine for which $H_i$ will it be the largest.

Now we show that we get the maximum for $i=\lfloor \frac{n-k}{2} \rfloor$ or $ \lceil \frac{n-k}{2} \rceil$. Let us take an integer $0 \leq i < \lfloor \frac{n-k}{2} \rfloor$ and show that $\omega_{\cF}(H_i) < \omega_{\cF}(H_{i+1})$ holds.
The two sums have the same number of terms. We define a bijection between them.

%If $c_1, \dots c_\ell \in H_i$ and $c_1, \dots c_\ell \in H_{i+1}$, then we can pairing them to each other. If $c_j \in H_i$, but $c_j \notin H_{i+1}$, then $c_j=i$. We pairing the term $T=\{c_1=i, c_2, \dots c_\ell \}$ to the  $T'=\{2i+k+1-c_\ell, 2i+k+1-c_{\ell-1},\dots, 2i+k+1-c_1=i+k+1 \}$. This is a bijective pairing such that $H \in H_i$, $H \notin H_{i+1}$, $H' \in H_{i+1}$ and $H \notin H_i$.

The terms without $i$ and $i+k+1$ appear identically in both sums, they can be matched to each other. The remaining terms are the ones with $i$ in $\omega_{\cF}(H_i)$ and the ones with $i+k+1$ in $\omega_{\cF}(H_{i+1})$. We match the term of $\{i=c_1<c_2< \dots <c_\ell \}$ to the term of $\{2i+k+1-c_\ell<2i+k+1-c_{\ell-1}< \dots <2i+k+1-c_1=i+k+1 \}$.

We can finish the proof with the following claim.

\begin{all}
Let $d=2i+k+1$ and $i=c_1 < c_2 < \dots < c_\ell \leq i+k$. If $i < \left\lfloor\frac{n-k}{2}\right\rfloor$, then $$\frac{n!}{(n-c_\ell)!(c_\ell-c_{\ell-1})!\dots (c_2-c_1)!c_1!} $$ $$< \frac{n!}{(n-(d-c_1))!((d-c_1)-(d-c_2))!\dots ((d-c_{\ell-1})-(d-c_\ell))!(d-c_\ell)!}$$
\end{all}

\begin{proof}
First, note that $d=2i+k+1\le 2\left(\frac{n-k}{2}-1\right)+k+1=n-1<n$.
Most of the factors in the denominators are same, so it is enough to prove that $(n-c_\ell)!c_1! > (n-(d-c_1))!(d-c_\ell)!$, or equivalently $\binom{n+c_1-c_\ell}{c_1}<\binom{n+c_1-c_\ell}{d-c_\ell}$. This is true since $c_1<d-c_\ell$ and $c_1+(d-c_\ell)<n+c_1-c_\ell$.
\end{proof}

%From the conditions we know that $c_1 < d-c_\ell$, $n-(d-c_1) < n-c_\ell$ and $(n-c_\ell) + c_1 = (n-(d-c_1)) +(d-c_\ell)$ which implies that the claim is true.

We matched the terms of $\omega_{\cF}(H_i)$ and $\omega_{\cF}(H_{i+1})$ such that every term of $\omega_{\cF}(H_{i+1})$ is at least as large as their pair and some of them are strictly larger, so $\omega_{\cF}(H_i) < \omega_{\cF}(H_{i+1})$.

With the same method one can show that if $i > \lceil \frac{n-k}{2} \rceil$, then $\omega_{\cF}(H_i) > \omega_{\cF}(H_{i+1})$, so $\omega_{\cF}(H_i)$ is maximal if $i = \lfloor \frac{n-k}{2} \rfloor$ or $i = \lceil \frac{n-k}{2} \rceil$.

If $n+k$ is even, equality holds only when all full chains intersect $\cF$ in $k+1$ sets of size between $\frac{n-k}{2}$ and $\frac{n+k}{2}$.
If $n+k$ is odd, equality holds only when all full chains intersect $\cF$ in $k+1$ sets, including all sets of size between $\frac{n-k+1}{2}$ and $\frac{n+k-1}{2}$, and exactly one of $\frac{n-k-1}{2}$ and $\frac{n+k+1}{2}$. Since this holds for all chains one can conclude that actually the family contains either all sets of size $\frac{n-k-1}{2}$ or all sets of size $\frac{n+k+1}{2}$.
Therefore the only extremal families are the ones mentioned in the theorem.
\end{proof}

Finally, we will examine a variant of Katona's Theorem for $\ell$-chains. We do not have a complete solution here and only mention some key observations.

Consider a family $\cF \subset 2^{[n]}$ such that there are no two sets $A,B\in\cF$ satisfying $A\subset B$ and $\vert B - A \vert < k$. Among such families we are looking for the one that has the most $\ell$-chains.

Since the condition on $\cF$ is chain-dependent, we can apply Theorem \ref{cdt2} as before to show that there is an optimal family that is the union of some full levels. The question is which levels give the most $\ell$-chains.

In Katona's Theorem $\{F\subset [n]~|~|F|\equiv \lfloor \frac{n}{2} \rfloor \pmod{k}\}$ is an extremal construction. %we get the maximum when a set $A \in \cF \Longleftrightarrow \vert A \vert \equiv \lfloor \frac{n}{2} \rfloor \pmod{k}$ or $\vert A \vert \equiv \lceil \frac{n}{2} \rceil \pmod{k}$.
Similar to the previous theorem, we were also interested in whether this would be an optimal family for the $\ell$-chains as well. However, we found many counterexamples to show that it is not optimal.

One is when $n=6$, $k=3$ and $\ell=2$. If we choose the family which gives the optimal solution for Katona's Theorem (all sets of size 0, 3 and 6), we get $41$ pieces of $\ell$-chains. However, by selecting all sets of size 1 and 4, we get $60$ pieces of $\ell$-chains.

Moreover, it is not even true that for optimal families $\cF$ there is an $x$ such that $A \in \cF \Longleftrightarrow \vert A \vert \equiv x \pmod{k}$. A counterexample to this is the following: $n=21$, $k=5$, $\ell=2$. In this case, the best solution is to select all sets of size 2, 7, 14 and 19.

{\bf Acknowledgements} The authors are thankful to G.O.H. Katona for the valuable discussions during the preparation of this paper and to the anonymous referees for their insightful remarks to improve the manuscript.

\nocite{*}
\bibliographystyle{plain}
% További stílusok:
% https://www.overleaf.com/learn/latex/Bibtex_bibliography_styles
\bibliography{revised3}

\end{document}